\documentclass[12pt]{article}

\usepackage{amssymb,amsmath,amsthm}
\usepackage{verbatim}
\usepackage{fullpage}

\newtheorem{theorem}{Theorem}
\newtheorem{lemma}{Lemma}

\newtheorem{corollary}{Corollary}

\newenvironment{definition}[1][Definition:]{\begin{trivlist}\item[\hskip \labelsep {\bfseries #1}]}{\end{trivlist}}
\newenvironment{remark}[1][Remark:]{\begin{trivlist}\item[\hskip \labelsep {\bfseries #1}]}{\end{trivlist}}

\newcommand{\bigS}{\mathcal{S}}
\newcommand{\R}{\mathbb{R}}
\newcommand{\Z}{\mathbb{Z}}

\newcommand{\N}{\mathbb{N}}

\newcommand{\G}{\mathcal{G}}

\title{Algebraically recurrent random walks on groups}
\author{Itai Benjamini \and Hilary Finucane \and Romain Tessera}

\begin{document}

\maketitle

\begin{abstract}
Initial steps are presented towards understanding which finitely generated groups are almost surely generated as a semigroup by the path of a random walk on the group.
\end{abstract}

\section{Introduction}

Let $G$ be a countable group, $\mu$ be a probability measure on $G$, $\zeta_i \sim \mu $ be i.i.d., and let $X_n = \zeta_1 \zeta_2 \cdots \zeta_n$. Then we call $(X_1, X_2, \ldots)$ a {\em $\mu$-random walk on $G$}. Since Furstenberg \cite{F} qualitative properties of  random walks on groups were used  to study and classify natural properties of groups $G$ or pairs $(G, \mu)$ such as the Liouville property and amenability; see e.g.\ \cite{KV}.
In this work we define a new group property based on random walks, which we call algebraic recurrence, and we present some initial steps towards understanding which groups are algebraically recurrent.

\begin{definition}
Let $(X_1, X_2, \ldots)$ be a $\mu$-random walk on $G$, and let $\bigS_n$ denote the semigroup generated by $\{X_n, X_{n+1}, \ldots\}$. We say $(G, \mu)$ is  {\em algebraically recurrent} (AR) if for all $n$, $\bigS_n= G$ almost surely, and we call $G$ AR if $(G, \mu)$ is AR for all symmetric measures $\mu$ with $\langle supp(\mu) \rangle = G$.
\end{definition}

Most classical properties of random walks on groups, such as recurrence/transience, Liouville property, etc.\ can be abstracted from the context of groups. By contrast, the definition of algebraic recurrence requires at least some binary operation on the state set of the random walk.

The use of a semigroup rather than a subgroup in the definition may seem unnatural, but in fact the property is trivial if defined in terms of subgroups rather than semigroups. To see this, let $\G_n$ denote the {\em group} generated by $\{X_n, X_{n+1}, \ldots\}$ and suppose $\langle supp(\mu) \rangle = G$. Then for each $g \in supp(\mu)$, $\Pr(g \notin \{\zeta_{n+1}, \zeta_{n+2}, \ldots \}) = 0$, but for all $i > n$, $ \zeta_i = (X_{i-1})^{-1}X_{i} \in \G_n$. Thus, $supp(\mu) \subset \G_n$, and so $\G_n = G$ almost surely. This argument in fact proves the more general fact:

\begin{lemma}\label{lemma.inverses} If $ X_i^{-1} \in \bigS_n$ almost surely for all $i \geq n$, then $\bigS_n = G$ almost surely. \end{lemma}

In this note we show that the class of AR groups is nontrivial: i.e.\ there exist AR and non-AR groups. For example, we prove that nilpotent finitely generated groups are AR, while free groups with more than 4 generators are not AR. We also prove that Liouville random walks on polycyclic groups are AR. By \cite{K}, this includes symmetric random walks with a first finite moment. We do not know if this fact extends to {\it any} symmetric random walk on a polycyclic group.

We have to admit our frustration at not being able to establish that the standard random walk on the free group with two (or 3) generators is not AR. It turns out that  Theorem \ref{theorem.freegroup} is by far the trickiest result of this paper and we would be curious to see another --more geometric-- proof of this fact.

In view of the fact that a non-centered random walk on $\Z$ is trivially not AR, the assumption that our  random walks are symmetric seems reasonable. However, one would legitimately be tempted to extend the results of this note to centered random walks: namely random walks whose projection to any cyclic quotient is centered. We leave this aspect of the question to future developments.

\

The paper is organized as follows: in Section \ref{section:AR}, we give examples of AR groups, then  Section \ref{section:nonAR} deals with the case of free groups. Finally Section \ref{section:Questions} is dedicated to open questions.

\section{Examples of AR groups}\label{section:AR}

\subsection{Torsion groups and lamplighters}
\label{section.torsion_Z}

Recall that $G$ is a torsion group if every element of $G$ has finite order. By Lemma~\ref{lemma.inverses}, any torsion group is AR, since in a torsion group, $X_i^{-1} = X_i^m$ for some $m$, and so $X_i^{-1}$ is in any semigroup that includes $X_i$. In fact, the following much stronger result holds.

\begin{theorem}\label{theorem.torsion} Suppose $H \triangleleft G$ is a torsion group. Then $G/H$ is AR if and only if $G$ is AR. \end{theorem}

\begin{proof}
Suppose $G/H$ is AR, and let $\pi$ denote the projection from $G$ to $G/H$. Let $(X_1, X_2, \ldots)$ be a $\mu$-random walk on $G$ with corresponding semigroup $\bigS_n$, and let $\bar{\bigS_n}$ be the semigroup of $G/H$ corresponding to the projected random walk $(\pi(X_1), \pi(X_2), \ldots)$. By Lemma~\ref{lemma.inverses}, to show that $G$ is AR, it suffices to show $X_i^{-1} \in \bigS_n$ for all $i \geq n$.

For any $X$ and $Y$ in $G$ with $\pi(X) = \pi(Y)^{-1}$, we have $XY \in H$, so there is an exponent $k$ such that $(XY)^k = e$, and thus $X^{-1} = Y(XY)^{k-1}$. So if there is a $Y_i$ in $\bigS_n$ with $\pi(Y_i) = \pi(X_i)^{-1}$, then we also have $X_i^{-1} \in \bigS_n$.

By the algebraic recurrence of $G/H$, we have $G/H = \bar{\bigS_n}$, and in particular, $\pi(X_i)^{-1} \in \bar{\bigS_n}$. But since $\bar{\bigS_n} = \pi(\bigS_n)$, this means that there is some $Y_i$ in $\bigS_n$ with $\pi(Y_i) = \pi(X_i)^{-1}$. So $G$ is AR.

Conversely, if $G$ is AR, then any random walk on $G/H$ can be lifted to a random walk on $G$; the corresponding semigroup is all of $G$, and so projects to all of $G/H$, showing that $G/H$ is AR.

\end{proof}

The lamplighter group of a group $G$, denoted $LL(G)$, is the wreath product $\Z/2\Z \ \wr \ G$. An element of $LL(G)$ is written $(x,f)$, where $x \in G$ and $f:G \rightarrow \Z/2\Z$, and $(x,f)(y,g) = (z,h)$, where $xy = z$ and $h(a) = f(a)g(ax^{-1})$.

Lamplighters often give examples of somewhat exotic behavior; for example, $LL(\Z)$ has exponential growth but is Liouville  and $LL(\Z^3)$ is amenable but non-Liouville \cite{KV}. It follows directly from Theorem~\ref{theorem.torsion} that lamplighters do not exhibit any unusual behavior in the case of algebraic recurrence: in fact, the algebraic recurrence of $LL(G)$ corresponds exactly to the algebraic recurrence of $G$.

\begin{corollary} $LL(G)$ is AR if and only if $G$ is AR. \end{corollary}
\begin{proof} The position function $pos(x,f) = x$ is a surjective homorphism from $LL(G)$ to $G$, and the kernel of $pos$ is a torsion group with exponent two.
\end{proof}

\subsection{Finitely generated abelian groups}

Lemma~\ref{lemma.inverses} can also be used to show that $\Z$ is AR. Indeed, by the symmetry of $\mu$, almost surely for all $n$ there will be $y^+, y^- \in \bigS_n$ with $y^+ > 0$ and $y^- < 0$. Then for each $i \geq n$, if $X_i > 0$ we can write $X_iy^- + (-y^- -1)X_i = -X_i$. But the LHS is in $\bigS_n$, so $-X_i \in \bigS_n$. Similarly, if $X_i < 0$ we have $(y^+-1)X_i + -X_iy^+ = -X_i$, and again the LHS is in $\bigS_n$ so $-X_i \in \bigS_n$. Using Lemma~\ref{lemma.inverses}, this suffices to show that $\Z$ is AR.

It is not much more difficult to see that $(\Z^d, \mu_d)$ is AR, when $\mu_d$ is uniform over the standard generating set. Indeed, after finitely many steps, the random walk will have visited $d$ linearly independent points, generating the intersection of a full-dimension lattice with a cone. Eventually, the random walk will visit a point $x$ that is in the opposite cone, and by adding arbitrarily large multiples of $x$, the entire lattice is in $\bigS_n$. Since there are only finitely many cosets of the lattice, the random walk eventually visits each coset, showing that all of $\Z^d$ is in $\bigS_n$.

However, this proof does not extend to arbitrary symmetric generating measures on $\Z^d$. For example, in $\Z^2$, $\mu$ could have a very heavy tail along the line $x=y$ and a very small weight along the line $x=-y$, so that there are cones that the random walk has non-zero probability never to intersect. So for the general case, a more subtle proof is needed.

Clearly, $(G,\mu)$ is AR if and only if the trace of a $\mu$-random walk on $G$ is almost surely not contained in any maximal subsemigroup of $G$. In the case of $\Z^d$, these maximal subsemigroups are easy to describe.

\begin{lemma}\label{lemma.Zd_subsemigroups} Every proper subsemigroup of $\Z^d$ is contained either in a proper subgroup of $\Z^d$ or in a half-space of $\Z^d$. \end{lemma}
\begin{proof}
Let $S$ be a subgroup of $\Z^d$. If $0$ is not in the convex hull of $S$, then there is a halfspace containing $S$. Otherwise, by Caratheodory's theorem, there are points $x_1, \ldots , x_{d+1}$ and positive numbers $t_1, \ldots , t_{d+1}$ such that $\sum t_i x_i = 0$ and $x_1, \ldots , x_d$ are linearly independent. Thus, $x_{d+1}$ is written as a linear combination of $x_1, \ldots x_d$, using only negative coefficients. This allows us to generate arbitrary linear combinations of $x_1, \ldots, x_d$, so the group $H$ generated by $x_1, \ldots , x_d$ is contained in $S$. Let $\bar{S}$ denote the projection of $S$ to $\Z^d/H$. Because $\Z^d/H$ is torsion, $\bar{S}$ is a subgroup. If $\bar{S} = \Z^d/H$, then $S = \Z^d$. Otherwise, $S$ is contained in a proper subgroup of $\Z^d$.
\end{proof}

\begin{definition}
Let $G$ be a countable group. Denote by $\mathfrak{S}(G)$ the set of subsemi-groups of $G$. Note that $\mathfrak{S}(G)$  is a compact space for the product topology (hence a standard Borel space).  The inverse of some semigroup $H$ is the semigroup consisting of inverses of elements of $H$.
Let $(H_n)$ be a decreasing sequence of $\mathfrak{S}(G)$-valued random variables. We shall say that $(H_n)$ is
\begin{itemize}

\item {\it non-degenerate} if $H_n$ generates $G$ as a subgroup for all $n$ almost surely;
\item {\it Liouville} if the tail $\sigma$-algebra is trivial;
\item {\it symmetric} if for every $n$, and every Borel subset $\Omega\subset \mathfrak{S}(G)$, the events $\{H_n\subset \Omega\}$ and $\{H_n^{-1}\subset \Omega\}$ are equiprobable. \end{itemize}
Clearly, if $(X_n)$ is a $\mu$-random walk with $\mu$ symmetric and non-degenerate, then the sequence $(\bigS_n)$ is symmetric non-degenerate. Moerover if $X_n$ is Liouville, then so is $\bigS_n$.
\end{definition}

\begin{theorem}\label{theorem.Zd} $\Z^d$ is AR for all $d\geq 1$. More generally, every non-degenerate Liouville symmetric decreasing sequence of random semigroups $(H_n)$ of $\Z^d$ is such that $H_n=\Z^d$ a.s. for all $n$.  \end{theorem}
\begin{proof}
The proof immediately follows from Lemma \ref{lemma.Zd_subsemigroups} together with the following lemma. \end{proof}
\begin{lemma}\label{Lem:half-space}
Let $G$ be a countable subgroup of $\R^d$ and let $(H_n)$ be some Liouville symmetric decreasing sequence of random semigroups of $G$, such that for all $n$, $H_n$ generates $\R^d$ as a vector space. Then a.s. $H_n$ does not eventually get trapped  in a closed half-space of $\R^d$.
\end{lemma}
\begin{proof} We will prove the lemma by induction on $d$, the case $d=0$ being trivial.
Let $A_n$ the closure of  the radial projection of $H_n$ to the sphere $S^{d-1}$ in $\R^d$. By compactness of the sphere, the intersection of the $A_n$'s is a non-empty closed subset $A \subset S^{n-1}$. \\

\noindent {\em Claim: $A$ is a deterministic set; i.e. there exists a set $T \subset S^{n-1}$ such that $\Pr(A = T) = 1$. Moreover, $A = -A$ almost surely.}\\
The symmetry of $A$ follows from the symmetry of $H_n$. To show that $A$ is deterministic, we use the Liouville property of $(H_n)$: $A$ depending only on the tail of $H_n$, any event depending only on $A$ has probability 0 or 1. But the only probability measure on closed sets that satisfies this property is the  Dirac measure; in particular, there is one closed set $T$ such that $P(A=T) = 1.$ \\

By the Claim, there is a pair of points $x, -x$ in $S^{d-1}$ that are almost surely contained in $A$. Thus,  $H_n$ is almost surely not contained in any halfspace that does not contain $x$ and $-x$. Let $\pi$ denote the projection onto the hyperplane orthogonal to $x$. If $H_n$ is eventually contained in halfspace containing $x$ and $-x$, then $\pi(H_n)$ must be contained in a halfspace of $\R^{d-1}$. But the projection $\pi(H_n)$ generates $\R^{d-1}$ as a vector space (and is obviously Liouville and symmetric). Hence the lemma follows by induction on the dimension.
\end{proof}

\begin{remark}
The level of generality of Theorem \ref{theorem.Zd} will be needed for the proof of the polycylic case (see Theorem \ref{thm:polycyclic}). \end{remark}

\subsection{Finitely generated nilpotent groups}
To prove algebraic recurrence of finitely generated nilpotent groups, we will need the following lemma.

\begin{lemma}\label{lemma.nilpotent_semigroups} Let $S$ be a subsemigroup of a torsion-free nilpotent group $N$, and let $\bar{S}$ be the projection of $S$ to $N/[N,N]$. Then $S = N$ if and only if $\bar{S} = N/[N,N]$. \end{lemma}
\begin{proof}
Let $S$ be a subsemigroup of $N$ that projects to all of $N/[N,N]$.  Let $1 = N_r \triangleleft N_{r-1} \triangleleft \cdots \triangleleft N_1 \triangleleft N_0 = N$ be the lower central series of $N$, and suppose by induction that the lemma holds for $N/Z$ for every cyclic subgroup $Z$ of $N_{r-1}$. For each such $Z$, $S$ projects to all of $(N/Z)/[(N/Z),(N/Z)]$, and so by induction it projects to all of $N/Z$. So to show $N \subset S$, it suffices to find a cyclic subgroup $Z \subset N_{r-1}$ such that $Z \subset S$.

Let $Z$ be an arbitrary cyclic subgroup of $N_{r-1}$, an let $z$ be a generator of $Z$. There are $a$ and $b$ in $N$ such that $[a,b] = z$, and by induction $S$ contains representatives of each coset of $Z'$, so there are $k_1, \ldots , k_4$ such that $c = az^{k_1}$, $d = bz^{k_2}$, $e = a^{-1}z^{k_3}$ and $f = b^{-1} z^{k_4}$ are all in $S$. A simple calculation shows that $c^nd^me^nf^m = z^{nm+L(n,m)}$, where $L(n,m) = (k_1+k_3)n + (k_2+k_4)m$, and $d^mc^nf^me^n = z^{-nm + L(n,m)}$. Letting $n$ and $m$ be large, we get that $z^k$ and $z^\ell$ are both in $S$, for some $k>0$ and $\ell < 0$. Together, $z^k$ and $z^\ell$ generate a cyclic subgroup $Z'$ of $Z$ which is contained in $S$.
\end{proof}

\begin{theorem} Every finitely generated nilpotent group $N$ is AR. \end{theorem}
\begin{proof} By Theorem~\ref{theorem.torsion}, we can assume that $N$ is torsion-free. Given a random walk $(X_1, X_2, \ldots)$ on $N$, let $(\bar{X_1}, \bar{X_2}, \ldots)$ denote the projection onto $N/[N,N]$. The semigroup generated by $\{\bar{X_n}, \bar{X_{n+1}}, \ldots\}$ is the projection $\bar{\bigS_n}$ of $\bigS_n$. By Theorem~\ref{theorem.torsion}, we can assume $N/[N,N] \cong \Z^d$, and so $N/[N,N]$ is AR, so $\bar{\bigS_n} = N/[N,N]$ almost surely. By Lemma~\ref{lemma.nilpotent_semigroups}, $\bigS_n = N$ almost surely.
\end{proof}

\subsection{Polycyclic groups}
To prove that finitely generated nilpotent groups are AR, we crucially used the fact that every symmetric random walk is Liouville. This is unknown for polycyclic groups. However one has
\begin{theorem}\label{thm:polycyclic}
Every Liouville symmetric non-degenerate random walk on a virtually polycyclic group is AR.
\end{theorem}
\begin{proof}
Let us start with an easy lemma
\begin{lemma}\label{lem:recAR}
To prove that a group $G$ is AR, it suffices to show that any finite index subgroup is AR. Moreover, the same holds if one restricts to Liouville random walks.
\end{lemma}
\begin{proof} If $H$ is a subgroup of $G$, then a $\mu$-random walk on $G$ can be projected to a random walk on $G/H$. If this random walk is recurrent, then we can define the harmonic measure $\mu_H$ on $H$. If $G/H$ is $\mu$-recurrent and $(H, \mu_H)$ is AR, then $(G,\mu)$ is AR. Indeed, the intersection of a $\mu$-random walk on $G$ with $H$ is a $\mu_H$-random walk on $H$ (which is Liouville if the latter is). Because $(H,\mu_H)$ is AR, we must have $H \subset \bigS_n$ almost surely. But by the recurrence of $G/H$, there is a representative of each coset of $H$ in $\bigS_n$. Thus, $G$ is in $\bigS_n$. As a special case of this fact, we see that if $H$ is a finite index subgroup of $G$ and $H$ is AR, then $G$ is AR.
\end{proof}
Now let us turn to the proof of Theorem \ref{thm:polycyclic}. Up to passing to a finite index subgroup, we can assume that $G$ is (finitely generated nipotent)-by-abelian. Let $(X_n)$ be a symmetric non-degenerate Liouville random walk on $G$ and let $(\bigS_n)$ be the corresponding sequence of semigroups. Since $G/[G,G]$ is AR, it is enough to prove that a.s.\ $H_n:=\bigS_n\cap [G,G]= [G,G]$. On the other hand, $[G,G]$ being nilpotent, up to dividing by the derived subgroup of $[G,G]$ (which is normal in $G$), one can assume that $[G,G]$ is abelian:\ this indeed follows from  Lemma \ref{lemma.nilpotent_semigroups}.

Up to dividing $G$ by a finite normal subgroup, and applying Theorem \ref{theorem.torsion}, one can assume that $[G,G]$ is torsion-free, hence isomorphic to $\Z^k$. Clearly $H_n$ is a Liouville symmetric decreasing sequence of random semigroups of $\Z^k$, so in order to apply Theorem \ref{theorem.Zd}, it is enough to show that $H_n$ is non-degenerate: this will end the proof of Theorem \ref{thm:polycyclic}.\\

\noindent {\em Claim: $H_n$ a.s.\ generates $[G,G]$ as a subgroup.}
\\ Observe that for every integer $m\in \N$, the subgroup $N_m$ of $m$-powers of elements of $[G,G]$ is normal in $G$ and let $\pi_m$ be the projection of $G$ to $G/N_m$. It follows from Theorem \ref{theorem.torsion} that a.s.\ $\pi_m(\bigS_n)=G/N_m$, and so $\pi_m(H_n)=[G,G]/N_m$. Recall that every proper subgroup of $\Z^k$ sits inside the kernel of some surjective morphism $\Z^k\to \Z/m\Z$. In particular such subgroup does not surject to  $G/N_m$ for the corresponding $m$. Put together, these two facts imply the claim, so the theorem.
\end{proof}

\section{The free group}\label{section:nonAR}

An example of a group that is not AR is the free group on more than four generators.

\begin{theorem} \label{theorem.freegroup} Let $F_d$ be the free group on $d$ generators, and let $\mu_d$ be uniform over the standard generating set of $F_d$. For $d >4$, $(F_d, \mu_d)$ is not AR. \end{theorem}

Let $X_n(i,j)$ denote the symbols $i$ through $j$ of $X_n$, written in its reduced form. Let $\log$ denote log base 2, and let $\exp$ denote exponentiation base 2. We will use the following lemma.

\begin{lemma}\label{lemma.small_branches} There exists a $j_0$ such that 2ith positive probability, for all $j>j_0$ there are at most $\log{j}$ strings of length $j$ that appear as prefixes of some $X_n$. \end{lemma}

\begin{proof} Consider a random walk $(Y_1, Y_2, \ldots)$ on $\Z^+$, reflected at 0, with probabilities $(2d-1)/(2d)$ and $1/(2d)$ to increase by one or decrease by one, respectively. The number of strings of length $j$ that appear as prefixes of some $X_n$ is upper bounded by the number $V_j$ of visits to $j$ in this biased random walk. The $V_j$ for $j \neq 0$ are i.i.d.\ geometric random variables with parameter $p = $ probability of returning to $j$. The return probability $p$ satisfies the equation $p = 1/(2d) + \left(\frac{1}{2d}\right) \left( \frac{2d-1}{2d} \right) p$, implying $p = \frac{2d}{(2d)^2 - 2d + 1}$. We have $\Pr(V_j > \log{j}) = p^{\log{j}}$, so the probability that there is a $j>j_0$ with $V_j > \log{j}$ is at most
$$ \sum_{j > j_0} p^{\log{j}} = \sum_{j>j_0} j^{\log{p}} < 1,$$
for sufficiently large $j_0$, since $p < 1/2$ and therefore $\log{p} < 1$ for $d > 4$.
\end{proof}

 Define:
$$A_r = \{w : |w| = r, w =X_{n_1}(i_1, j_i)\cdots X_{n_m}(i_m,j_m), i_1 = 1, i_k \leq \log j_{k-1} \mbox{ for all } 2 \leq k \leq m\}.$$
We will show that there exists an $n_0$ such that with constant probability, all words in $\bigS_{n_0}$ have length-$r$ prefixes in $A_r$.
\begin{lemma} $|A_r| \leq 4^r$ with positive probability. \end{lemma}
\begin{proof} Let $\ell_k = j_k-i_k$. There are $2^{r-1}$ ways to choose the $\ell_k$, so we only need to show that for any fixed $\{\ell_k\}$, there are fewer than $2^r$ ways to choose the $\{i_k\}$ and $\{X_{n_k}\}$. We have $i_k \leq \log{j_{k-1}} = \log{(i_{k-1}+\ell_{k-1})}$, so the number of ways to choose the $i_k$ is at most
\begin{align*}
&\log{j_1}\cdots \log_{j_m}\\
 \leq & \log{\ell_1}\cdot \log{(\ell_2 + \log{\ell_1})} \cdot \log{(\ell_3 + \log{(\ell_2 + \log{\ell_1})})} \cdots \cdots \log{(\ell_m + \log{(\ell_{m_1} + \cdots + \log\log\cdots\log{\ell_1})})}\\
 \leq & \Pi_{1 \leq k \leq m} (\log{\ell_k} + \log\log{\ell_{k-1}} + \cdots + \log \cdots\log \ell_1)\\
 = & \exp\left(\sum_{1 \leq k \leq m}\log (\log{\ell_k} + \log\log{\ell_{k-1}} + \cdots + \log \cdots\log \ell_1)\right)\\
 \leq & \exp\left(\sum_{1 \leq k \leq m} (\log \log{\ell_k} + \log \log\log{\ell_{k-1}} + \cdots + \log \cdots\log \ell_1)\right)\\
 \leq & \exp\left(\sum_{1 \leq k \leq m} (\log \log{\ell_k} + \log \log\log{\ell_{k}} + \cdots + \log \cdots\log \ell_k + 1)\right)\\
 \leq & \exp\left(2 \sum_{1 \leq k \leq m} \log \log{\ell_k}\right)\\
 \leq & \exp(2 (r/4))\\
  =& 2^{r/2}.
\end{align*}
where  the last inequality follows because $\sum_{1 \leq k \leq m} \log \log{\ell_k}$ is maximized when all of the $\ell_k$ are equal to 4. Similarly, Lemma~\ref{lemma.small_branches} tells us that with positive probability, the number of possible choices of the $\{X_{n_k}\}$ is also bounded by $\log{j_1} \cdots \log{j_m} \leq 2^{r/2}$. Thus, the we have $|A_r|<4^r$, as claimed.
\end{proof}
Let $A = \cup_r A_r$ and $d>4$. For reduced words $x$ and $y$, let $cancel(x,y)$ denote the number of symbols that are cancelled in the multiplication $x\cdot y$.

\begin{lemma}\label{lemma.A_r_preserved} There exists an $n_0$ such that with positive probability,  $cancel(X_n, w) < \log{|X_n|}$ for any $n>n_0$ and any $w \in A$. \end{lemma}

\begin{proof} For a random $X$ with $|X| = s$ and a fixed $w$, we have $\Pr(cancel(X,w) > \log{s}) \leq  (2d-1)^{-\log{s}}$ (times a factor $\frac{2d-1}{2d}$ that we will ignore). For any $w \in A_r$ with $r > \log{n}$, the length $\log{n}$ prefix of $w$ is in $A_{\log{n}}$. So the probability that there exists a $w \in A$ such that $cancel(X,w) > \log{s}$ is at most $|A_{\log{s}}|(2d-1)^{-\log{s}} \leq s^{-\log((2d-1)/4)}$ by a union bound. When $d > 4$, we have $-\log((2d-1)/4) < -1$. Union bounding this over the at most $\log{s}$ choices of $X_n$ for which $|X_n| = s$, and then over all $s \geq s_0$, we get $\sum_{s \geq s_0} ( \log{s} ) s^{-\log((2d-1)/4)}$. Because the sum converges, we can choose $s_0$ large enough that this sum is strictly less than one. By the transitivity of the random walk, there exists an $n_0$ so that with positive probability, $|X_n| > s_{min}$ for all $n>n_0$.
\end{proof}

\begin{proof}[Proof of Theorem~\ref{theorem.freegroup}]  The proof is by induction. $X_n \in A$ by definition for all $n$. By Lemma~\ref{lemma.A_r_preserved}, with positive probability, if $X_{n_1} \cdots X_{n_m} \in A$ for all $n_1, \ldots n_m \geq n_0$, then $X_{n_1} \cdots X_{n_{m+1}} \in A$ for all $n_1, \ldots n_{m+1} \geq n_0$. Thus, $\bigS_{n_0} \subset A \neq F_d$.
\end{proof}

\section{Open Questions}\label{section:Questions}
This first study leaves several natural questions open.
\begin{itemize}
\item {\em Is $(G, \mu)$-AR a group property?} We do not know if it is possible for there to be two symmetric measures $\mu_1$ and $\mu_2$ on a group $G$ with $\langle \mu_1 \rangle = \langle \mu_2 \rangle = G$ such that $(G, \mu_1)$ is AR and $(G, \mu_2)$ is not AR. A first step towards determining whether this is possible could be to prove that there is no $\mu$ for which $(F_d, \mu)$ is AR.

\item {\em Is AR preserved under taking finite index subgroups?} It follows from Lemma \ref{lem:recAR} that if $G$ contains a finite-index AR subgroup, then it is AR. We do not know if the converse holds.

\item {\em Non-AR groups.} We strongly suspect that the free group on two generators in not AR, but our proof technique is not strong enough to show this, and it does not follow immediately from the fact that $F_d$ is a finite index subgroup of $F_2$ (see the previous question). On the other hand, perhaps the proof of Theorem~\ref{theorem.freegroup} extends to small cancellation groups with growth at least $10^n$. More generally, it would be interesting to prove that small cancellation groups or hyperbolic groups are not AR. (There are nonamenable torsion groups  \cite{Ol}, ruling out the possibility that no nonamenable groups are AR.)

\item {\em What is the structure of the sequence of semigroups in non-AR groups?} For non-AR groups, we know that $\bigS_n$ is not all of $G$, but it would be interesting to determine other properties of $\bigS_n$. For example, what is the growth of $\bigS_n$? Is $\bigS_n$ transient? Is the intersection $ \cap_n \bigS_n$ empty almost surely? In particular what do the limit sets of such semigroups look like in free groups?

\item{\em Are Liouville random walks algebraically recurrent?} Recall that the converse is false (see Section \ref{section.torsion_Z}).

\item{\em Infinitely generated groups.}
Infinitely generated groups present a separate challenge. For example, we do not know if the abelian group $\oplus_\Z \Z$ is AR.

\item{\em Quantitative versions.}
It follows from Theorem 3.2 in \cite{KXZ} that, letting $R$ denote the range of Brownian motion in $\R^n$, $\R^n$ is almost surely covered by $R+R+...+R$,   $\lceil n/2 +1  \rceil$ times. A simpler second moment argument gives a similar statement with $\lceil d/2  \rceil$  for $\Z^d$ or nilpotent groups. This suggests that quantitative variants of algebraic recurrence could be interesting. For example, when $(G, \mu)$ is AR, estimate the probability $\bigS_n = G$.


\end{itemize}

\noindent \textbf{Acknowledgements:} Thank you to Elon Lindenstrauss and Ron Peled for pointing out errors in earlier versions of the proofs of Theorems~\ref{theorem.Zd} and \ref{theorem.freegroup}, respectively. Thank you also to Uri Bader, Vadim Kaimanovich and Yehuda Shalom for useful comments.

\end{document}